%% file: GrafiTorici.tex
\documentclass[a4paper,11pt,reqno]{amsart}
\usepackage[english]{babel}
\usepackage[T1]{fontenc}
\usepackage[utf8]{inputenc}
\usepackage{amsmath}
\usepackage{amssymb}
\usepackage{amsthm}
\usepackage{mathtools}
\usepackage{booktabs}
\usepackage{enumitem}
\usepackage{multirow}
\usepackage{graphicx}
\usepackage[caption=false]{subfig}
\usepackage[rgb]{xcolor} % needed for hsb colors in tikzpictures
\usepackage{tikz}
\usepackage{csquotes}
\usepackage{todonotes}
\usepackage[style=numeric-comp,backend=biber,sorting=nyt]{biblatex}

\usepackage{etoolbox}

\newcommand{\NN}{\mathbb{N}} % natural numbers
\newcommand{\ZZ}{\mathbb{Z}} % integer numbers
\newcommand{\CC}{\mathbb{C}} % complex numbers
\newcommand{\PP}{\mathbb{P}} % projective space
\newcommand{\tgsp}{\mathrm{T}} % notation for tangent space
\newcommand{\arr}{\mathcal{A}} % arrangement
\newcommand{\art}{\arr^T} % arrangement
\newcommand{\arh}{\arr^H} % arrangement
\newcommand{\TT}{T} % torus
\newcommand{\kk}{\mathcal{K}} % layer
\newcommand{\MM}{\mathcal{M}} % basis
\newcommand{\GG}{\mathcal{G}} % building set
\newcommand{\NS}{\mathcal{S}} % nested set
\newcommand{\cpl}{\mathcal{M}} % complement
\newcommand{\VV}{\mathcal{V}} % set of vertices of a graph
\newcommand{\EE}{\mathcal{E}} % set of edges of a graph
\newcommand{\st}{\mid} % 'such that', to be used for sets, spans etc.
\newcommand{\card}[1]{\#{#1}}%{\left|#1\right|} % cardinality
\newcommand{\rest}[1]{|_{#1}} % restriction
\newcommand{\Bl}[2]{\operatorname{Bl}_{#1}{#2}} % blowup

\newcommand{\YY}{\mathcal{Y}}
\newcommand{\Yh}{\mathcal{Y}^H}
\newcommand{\Yt}{\mathcal{Y}^T}

\newcommand{\dt}[1]{\widetilde{#1}} % dominant transform; maybe change notation?
\newcommand{\play}{\mathcal{C}} % poset of layers

\newcommand{\susp}[1]{\widehat{#1}}

\DeclareMathOperator{\supp}{supp}
\DeclareMathOperator{\codim}{codim}
\DeclareMathOperator{\poin}{Poin}

\renewcommand{\bar}{\overline}

\DeclareMathOperator{\inv}{inv}
\DeclareMathOperator{\lec}{lec}

\theoremstyle{definition}
\newtheorem{defn}{Definition}[section]

\theoremstyle{plain}
\newtheorem{thm}[defn]{Theorem}
\newtheorem{lemma}[defn]{Lemma}
\newtheorem{prop}[defn]{Proposition}
\newtheorem{crl}[defn]{Corollary}

\theoremstyle{remark}
\newtheorem{rmk}[defn]{Remark}

\newtheorem{exam}[defn]{Example}

\makeatletter
\let\@@pmod\pmod
\DeclareRobustCommand{\pmod}{\@ifstar\@pmods\@@pmod}
\def\@pmods#1{\mkern5mu({\operator@font mod}\mkern 6mu#1)}
\let\@@mod\mod
\DeclareRobustCommand{\mod}{\@ifstar\@mods\@@mod}
\def\@mods#1{\mkern5mu{\operator@font mod}\mkern 6mu#1}
\makeatother

\usetikzlibrary{calc}

\numberwithin{equation}{section}

\makeatletter
\patchcmd{\@setaddresses}{\indent}{\noindent}{}{}
\patchcmd{\@setaddresses}{\indent}{\noindent}{}{}
\patchcmd{\@setaddresses}{\indent}{\noindent}{}{}
\patchcmd{\@setaddresses}{\indent}{\noindent}{}{}
\makeatother

\title[An isomorphism between models of graphic arrangements]{An isomorphism between projective models of toric and hyperplane graphic arrangements}
\author{Giovanni Gaiffi}
\address{Giovanni Gaiffi\newline
Dipartimento di Matematica\newline
Università di Pisa\newline
Largo B.\ Pontecorvo, 5\newline
56127 Pisa, Italy}
\email{giovanni.gaiffi@unipi.it}
\author{Oscar Papini}
\address{Oscar Papini\newline
Istituto di Scienza e Tecnologie dell’Informazione ``A. Faedo''\newline
Consiglio Nazionale delle Ricerche\newline
Via G.\ Moruzzi, 1\newline
56124 Pisa, Italy}
\email{oscar.papini@isti.cnr.it}
\author{Viola Siconolfi}
\address{Viola Siconolfi\newline
Dipartimento di Meccanica, Matematica e Management\newline
Politecnico di Bari\newline
Via E.\ Orabona, 4\newline
70126 Bari, Italy}
\email{viola.siconolfi@poliba.it}

\date{}

\begin{document}

\maketitle

\begin{abstract}
This paper presents a bridge between the theories of wonderful models associated with toric arrangements and wonderful models associated with hyperplane arrangements. In a previous work, the same authors noticed that the model of the toric arrangement of type $A_{n-1}$ is isomorphic to the one of the hyperplane arrangement of type $A_{n}$; it is natural to ask if there exist similar isomorphisms between other families of arrangements. The aim of this paper is to study one such family, namely the family of arrangements defined by graphs. The main result states that there is indeed an isomorphism between the model of a toric arrangement defined by a graph $\Gamma$ and the model of a hyperplane arrangement defined by the cone of $\Gamma$, provided that a suitable building set is chosen.

\end{abstract}

\section{Introduction}
In this paper we highlight a bridge between wonderful models of toric arrangements and wonderful models of subspace arrangements. In particular we point out  some cases where  the  wonderful model of a toric arrangement is isomorphic to the wonderful model of an associated subspace arrangement. In these cases both the involved arrangements are graphic arrangements. 

\subsection{Historical sketches}
A \emph{subspace arrangement} is a collection of subspaces in a vector space $V$. Analogously, a \emph{toric arrangement} is a collection of subtori in an algebraic torus $T$. In both cases, in this paper we will assume that the base field is $\CC$. We say that an arrangement is divisorial if all its elements have codimension 1. A \emph{wonderful model} for an arrangement is a variety whose big open dense set is isomorphic to the complement of the arrangement and whose boundary is a divisor with normal crossings. More precise definitions will be provided in Section~\ref{sec:models}.

The construction of wonderful models of subspace arrangements was first described in the seminal papers~\cite{wonderful1,wonderful2}. The initial motivation was the study of Drinfeld's construction in~\cite{drinfeld} of special solutions of the Knizhnik-Zamolodchikov equations with some prescribed asymptotic behavior, but it was soon pointed out that the models are geomtric objects with their own great interest. For instance in the case of a complexified root arrangement of type $A_n$ the minimal model coincides with the moduli spaces of stable curves of genus 0 with $n+2$ marked points. 

In~\cite{wonderful1} De Concini and Procesi showed, using a description of the cohomology rings of the projective wonderful models to give an explicit presentation of a Morgan algebra, that the mixed Hodge numbers and the rational homotopy type of the complement of a complex subspace arrangement depend only on the intersection lattice (viewed as a ranked poset). The cohomology rings of the models of complex subspace arrangements were also studied in~\cite{yuzvinsky,gaiffiselecta} and, in the real case, in \cite{ehkr,rains}. The case of arrangements associated with complex reflection groups was studied from different point of views in~\cite{hendersonwreath}  and in~\cite{callegarogaiffilochak}. 
 
The connections between the geometry of these models and the Chow rings of matroids were pointed out first in~\cite{chow} and then in~\cite{adiprasito}, where they also played a crucial role in the study of some relevant log-concavity problems. The relations with toric and tropical geometry were enlightened for instance in~\cite{feichtnersturmfels,denham,amini}.
 
The study of toric arrangements started in~\cite{looijenga}.  In~\cite{deconciniprocesivergne} and~\cite{DCP3} the role of toric arrangements as a link between partition functions and box splines is pointed out. In~\cite{CDDMP}, it was shown, extending the results in~\cite{calledelu,callegaro2019erratum} and~\cite{Pagariatwo}, that the data needed in order to state the presentation of the rational cohomology ring of the complement $\MM(\arr)$ of a toric arrangement $\arr$ is fully encoded in the poset given by all the connected components of the intersections of the layers. It follows that in the divisorial case the combinatorics of this poset determines the rational homotopy of $\MM(\arr)$.

One of the motivations for the construction of \emph{projective wonderful models of a toric arrangement} $\arr$ in~\cite{deconcgaiffi1}, in addition to the interest in their own geometry, was that they could be an important tool to explore the generalization of the above mentioned results to the non-divisorial case.

Indeed the presentation of the cohomology ring of these models described in~\cite{deconcgaiffi2} was used in~\cite{mocipaga} to construct a Morgan differential algebra which determines the rational homotopy type of $\MM(\arr)$. We notice that these models, and therefore their associated Morgan algebras, depend not only on the initial combinatorial data, but also on some choices. In~\cite{deconcgaiffi3} a new differential graded algebra was constructed as a direct limit of the above mentioned differential Morgan algebras: it has a presentation which depends only on a set of initial discrete data extracted from $\arr$, and it can be used to prove that in the non-divisorial case the rational homotopy type of $\MM(\arr)$ depends only on these data.

\subsection{A bridge between the two families of models}
It is well-known that a root system defines both a hyperplane arrangement and a toric arrangement, so it is natural to study the projective wonderful models associated with them. In~\cite{gps}, the authors noticed that the model of the toric arrangement of type $A_{n-1}$ is isomorphic to the one of the hyperplane arrangement of type $A_n$. In this work we extend this result 
by studying a wider family of arrangements, namely the graphic arrangements (see Definitions~\ref{defn_hyp_graph_arr} and \ref{defn_tor_graph_arr}). In particular, given a graph $\Gamma$ we consider the toric arrangement defined by $\Gamma$ and the hyperplane arrangement defined by the cone of $\Gamma$ (see Definition~\ref{def:conegraph}). It turns out that, by choosing a convenient building set, the wonderful models associated with the two arrangements are isomorphic. The proof that we present relies on two key observations:
\begin{enumerate}
\item both their complements can be realised as the complement of the same projective hyperplane arrangement in a suitable projective space;
\item the constructions of both the toric and the hyperplane models can be translated as the construction of the model associated with this projective hyperplane arrangement, with the only difference being the order of the blowups of the elements of the building set.
\end{enumerate}

\subsection{Structure of the paper}
The paper is structured as follows: in Section~\ref{sec:models} we recall the definitions and constructions related to projective models associated with hyperplane and toric arrangements; in Section~\ref{sec:graphic_arr} we introduce the graphic arrangements and study the construction of the models associated with the toric and hyperplane graphic arrangements, focusing our attention on the choice of the building set; Section~\ref{sec:isomorphism_models} is devoted to our main result, namely the proof of the isomorphism between the model of the toric arrangement associated with a graph $\Gamma$ and the one of the hyperplane arrangement associated with the cone of $\Gamma$; Section~\ref{sec:examples} concludes the paper by showing how our main result can be applied to particular families of graphic arrangements to uncover interesting equalities between known combinatorial objects.
\begin{rmk}
We will use the two superscripts $(\ )^T$ and $(\ )^H$ to denote the objects relative to the fields of toric and hyperplane arrangements respectively.
\end{rmk}

\section{Compact models of hyperplane and toric arrangements}\label{sec:models}
Both hyperplane/subspace arrangements and toric ones can be seen as special cases of \emph{arrangements of subvarieties}. Following~\cite{li}, in this section we recall the main definitions and constructions of projective wonderful models in this more general setting, and afterwards we translate them in the context of the hyperplane and toric arrangements.
%In this section we recall the main definitions and constructions for the more general concept of projective wonderful models for arrangements of subvarieties, and afterwards we specialize in the two cases of hyperplane arrangements and toric arrangements.
\begin{defn}[{see~\cite[Definition~2.1]{li}}]\label{def:arrsubvar}
Let $X$ be a non-singular algebraic variety. An \emph{arrangement of subvarieties} of $X$ is a finite set $\Lambda$ of non-singular closed connected subvarieties properly contained in $X$ such that
\begin{enumerate}
\item for every two $\Lambda_i,\Lambda_j\in\Lambda$, either $\Lambda_i\cap\Lambda_j$ is a disjoint union of elements of $\Lambda$ or $\Lambda_i\cap\Lambda_j=\emptyset$;
\item if $\Lambda_i\cap\Lambda_j\neq\emptyset$, the intersection is \emph{clean}, i.e.\ it is non-singular and for every $y\in\Lambda_i\cap\Lambda_j$ we have the following conditions on the tangent spaces:
\[
\tgsp_y(\Lambda_i\cap\Lambda_j)=\tgsp_y(\Lambda_i)\cap\tgsp_y(\Lambda_j).
\]
\end{enumerate}
\end{defn}
We denote by $\cpl(\Lambda)$ the complement of the arrangement $\Lambda$ in $X$, i.e.
\[
\cpl(\Lambda)\coloneqq X\setminus\bigcup_{\Lambda_i\in\Lambda}\Lambda_i.
\]

If for every two $\Lambda_i,\Lambda_j\in\Lambda$ the intersection $\Lambda_i\cap\Lambda_j$ is either empty or connected, the arrangement is called \emph{simple}. For the rest of this paper we will only deal with simple arrangements, therefore from now on this hypothesis will always be implicitly assumed.

The key ingredient for the construction of the model is the notion of building set, of which we give two definitions (see~\cite[Definition~2.2]{li}).
\begin{defn}\label{def:building_arr}
Let $\Lambda$ be an arrangement of subvarieties. A subset $\GG\subseteq\Lambda$ is a \emph{building set} for $\Lambda$ if for every $L\in\Lambda\setminus\GG$ the minimal elements (w.r.t.\ the inclusion) of the set $\{G\in\GG\st L\subset G\}$ intersect transversally and their intersection is $L$. These minimal elements are called the \emph{$\GG$-factors} of $L$.
\end{defn}
\begin{defn}\label{def:building}
Let $\GG$ be a set of connected, closed, non-singular subvarieties of a variety $X$ and let $\Lambda(\GG)$ be the set of all the connected components of all the possible non-empty intersections of elements of $\GG$ (i.e.\ the arrangement of subvarieties \emph{induced} by $\GG$). We say that $\GG$ has the property of being \emph{building} if it is building for $\Lambda(\GG)$ according to Definition~\ref{def:building_arr}.
\end{defn}
\begin{rmk}
Definition~\ref{def:building} remarks that being building is an intrinsic property of a set of subvarieties. In fact, while the induced arrangement is unique for a chosen set of subvarieties $\GG$, there may exist different building sets for a chosen arrangement $\Lambda$---however, if $\GG$ is building for $\Lambda$, then $\GG$ is also building according to Definition~\ref{def:building} since in this case $\Lambda(\GG)=\Lambda$.
\end{rmk}
Finally, starting from a non-singular variety $X$, an arrangement of subvarieties $\Lambda$ and a building set $\GG$ for $\Lambda$, consider the locally closed embedding
\begin{equation}\label{eq:model_embedding}
\cpl(\Lambda) \longrightarrow \prod_{G\in \GG}\Bl{G}{X}
\end{equation}
where $\Bl{G}{X}$ is the blowup of $X$ along $G$.
\begin{defn}[{see~\cite[Definition~1.1]{li}}]\label{def:wonderfulmodel}
The closure of the image of the morphism~\eqref{eq:model_embedding} is the \emph{wonderful model} associated with $X$, $\Lambda$ and $\GG$ and it is denoted by $\YY(X,\Lambda,\GG)$. If $\Lambda$ is the arrangement induced by $\GG$ we will simply write $\YY(X,\GG)$.
\end{defn}

%There are alternative constructions of this model; in the following we will use the one provided by Li~\cite{li}.
\begin{thm}[{see~\cite[Theorem~1.3]{li}}]\label{thm:iterblowup}
Let $\GG$ be a building set in an algebraic variety $X$. Let us order the elements $G_1,\dotsc,G_m$ of $\GG$ in such a way that for every $1\leq k \leq m$ the set $\GG_k\coloneqq\{G_1,\dotsc,G_k\}$ is building. Then if we set $X_0\coloneqq X$ and $X_k\coloneqq \YY(X,\GG_k)$ for $1\leq k\leq m$, we have
\[
X_k=\Bl{\dt{G_k}}{X_{k-1}},
\]
where $\dt{G_k}$ denotes the dominant transform of $G_k$ in $X_{k-1}$. In particular $\YY(X,\GG)$ is obtained for $k=m$.
\end{thm}
\begin{rmk}
\begin{enumerate}
\item Any total ordering of the elements of a building set $\GG=\{G_1,\dotsc,G_m\}$ which refines the ordering by inclusion, that is $i<j$ if $G_i\subset G_j$, satisfies the condition of Theorem~\ref{thm:iterblowup}.
\item It follows from the preceding constructions that $\cpl(\Lambda)$ is the open dense set of $\YY(X,\Lambda,\GG)$ and that the boundary of $\YY(X,\Lambda,\GG)$ is the union of the non-singular irreducible divisors $D_G$ provided by the transforms of every $G\in\GG$. The intersection of any subset of these divisors is non-empty if and only if the corresponding subset of $\GG$ is $\GG$-nested (see Definition~\ref{def_nested_set_simple}). If this intersection is non-empty, then it is transversal. (See also~\cite[Theorem~1.2]{li}.)
\end{enumerate}

\end{rmk}

\subsection{Models for hyperplane arrangements}
Let $V=\CC^n$ as a vector space. An arrangement of subspaces $\arr$ is a finite family of linear subspaces in $V$. We point out that $\arr$ is central, i.e.\ $0$ belongs to every subspace of $\arr$, and we can assume that $\arr$ is essential, i.e.\ $\bigcap_{H\in\arr}H=\{0\}$.

The poset of intersections of $\arr$ is the set of all the possible intersections of the elements of $\arr$, including the whole space $V$ obtained as the empty intersection, partially ordered by reverse inclusion. It is denoted by $L(\arr)$ and it is a lattice (because $\arr$ is central).

Since $\arr$ is central, its projectivization $\bar{\arr}$ is well-defined as an arrangement of projective subspaces in $\PP(V)=\PP^{n-1}$. Given a building set $\GG$ for $L(\arr)$, a projective wonderful model associated with $\arr$ is the closure of the image of the locally closed embedding
\[
\cpl(\bar{\arr})\longrightarrow \PP(V)\times\prod_{G\in\GG}\PP(V/G)
\]
(see~\cite{wonderful1}) and we denote this model as $\YY^H(\arr,\GG)$. This construction is equivalent to the one in Definition~\ref{def:wonderfulmodel} by noticing that $L(\bar{\arr})$ forms an arrangement of subvarieties in $\PP^{n-1}$ according to Definition~\ref{def:arrsubvar}.

%and its poset of intersections $L(\bar{\arr})$, which is isomorphic to $L(\arr)\setminus\{0\}$, forms an arrangement of subvarieties in the sense of Definition~\ref{def:arrsubvar}. Therefore, given a building set $\GG$ for $L(\arr)$ not containing $\{0\}$ (this can always be assumed: if $\{0\}\in\GG$ we can consider $\GG\setminus\{0\}$ which is still building for $L(\arr)$),\todo{e' questo cio' che intendeva Giovanni?} we can define a projective wonderful model associated with $\arr$ as in Definition~\ref{def:wonderfulmodel} where $X=\PP^{n-1}$, $\Lambda=L(\bar{\arr})$ and $\bar{\GG}$ as building set, where $\bar{\GG}\subset L(\bar{\arr})$ is the set of the projectivizations of the elements of $\GG$. We denote this model as $\YY^H(\arr,\GG)$. This is equivalent to the classical definition~\cite{wonderful1} of the model as the closure of the image of the locally closed embedding
%\[
%\cpl(\bar{\arr})\longrightarrow \PP(V)\times\prod_{G\in\GG}\PP(V/G).
%\]

%; moreover $L(\arr)$ forms an arrangement of subvarieties in $V$ according to Definition~\ref{def:arrsubvar}.

\subsection{Models for toric arrangements}\label{subsec:models_toric}
Let $\TT=(\CC^*)^n$ be an algebraic complex torus and let $X^*(\TT)$ be its group of characters. For $\chi\in X^*(\TT)$, let $x_\chi\colon\TT\to\CC^*$ be the corresponding character on $\TT$. A \emph{layer} in $\TT$ is a subvariety of $\TT$ of the form
\begin{equation}\label{eq:layer_def}
\kk(\Gamma,\phi)\coloneqq\{t\in\TT\mid x_\chi(t)=\phi(\chi)\text{ for all } \chi\in\Gamma\}
\end{equation}
where $\Gamma<X^*(\TT)$ is a split direct summand and $\phi\colon\Gamma\to\CC^*$ is a homomorphism. A toric arrangement $\arr$ is a (finite) set of layers $\{\kk_1,\dotsc,\kk_r\}$ in $\TT$.

The poset of layers of $\arr$ is the set of all the connected components of the possible intersections of the elements of $\arr$, including the whole torus $\TT$ obtained as the empty intersection, partially ordered by reverse inclusion. It is denoted by $\play(\arr)$.

Following~\cite{deconcgaiffi1}, in order to construct a projective wonderful model associated with $\arr$, first we embed the torus $\TT$ in a suitable toric variety $X_\Delta$ with associated fan $\Delta$. In particular, a toric variety $X_\Delta$ is good for $\arr$ if every layer of $\play(\arr)$ has an equal sign basis with respect to the fan $\Delta$.

Let $\play(\bar{\arr})$ be the set of the closures of the layers of $\play(\arr)$ in $X_\Delta$; it can be proven that this set forms an arrangement of subvarieties in $X_\Delta$ according to Definition~\ref{def:arrsubvar}; therefore, given a building set $\GG$ for $\play(\bar{\arr})$, we can define a projective wonderful model associated with $\arr$ according to Definition~\ref{def:wonderfulmodel}. We denote this model as $\YY^T(\arr,\GG)$, where the dependence on $X_\Delta$ is left implicit.

\section{Graphic arrangements and their models}\label{sec:graphic_arr}
In this section we present the hyperplane and toric arrangements associated with some graphs and introduce the ingredients to build the projective wonderful models that we are going to study in Section~\ref{sec:isomorphism_models}.

We identify a graph with the pair $(\VV,\EE)$ where $\VV$ is its set of vertices and $\EE\subset \VV\times \VV $ is its set of edges. We assume all graph to be simple, undirected and without loops. Moreover we will consider only graphs with at least two vertices.

A key concept that will be widely used in this work is that of the cone of a graph. For convenience we report here the definition.
\begin{defn}\label{def:conegraph}
Let $\Gamma=(\VV,\EE)$ be a graph with $\VV=[n]\coloneqq\{1,\dotsc,n\}$. The \emph{cone} of $\Gamma$, denoted with $\susp{\Gamma}$, is a graph with vertices $\susp{\VV}=[n]\cup\{0\}$ and edges
\[
\susp{\EE}=\EE\cup\{(0,j)\st j\in [n]\}.
\]
\end{defn}

In the following definitions, let $v=(1,\dotsc,1)$ of suitable length.
\begin{defn}\label{defn_hyp_graph_arr}
Identify $\CC^{n-1}$ with $\CC^n/(\CC v)$. Given a graph $\Gamma=(\VV,\EE)$ with $\VV=[n]$, we define the associated hyperplane arrangement $\arh(\Gamma)$ in $\CC^{n-1}$ as the set of hyperplanes $\{H_{ij}\}_{(i,j)\in \EE}$ where
 $H_{ij}=\{(x_1,\ldots,x_n)\in\CC^n \st x_i=x_j\}/(\CC v)$.
\end{defn}

\begin{defn}\label{defn_tor_graph_arr}
Identify $(\CC^*)^{n-1}$ with $(\CC^*)^n/(\CC^* v)$. Given a graph $\Gamma=(\VV,\EE)$ with $\VV=[n]$, we define the associated toric arrangement $\art(\Gamma)$ in $(\CC^*)^{n-1}$ as the set of layers
$\{\kk_{ij}\}_{(i,j)\in \EE}$ where
 $\kk_{ij}=\{(t_1,\ldots,t_n)\in(\CC^*)^n \st t_it_j^{-1}=1\}/(\CC^* v)$.
\end{defn}
Notice that $\kk_{ij}$ is a layer according to~\eqref{eq:layer_def}---in fact, in this case we have that $\kk_{ij}=\kk(\Gamma_{ij},\phi_{ij})$, where
\begin{itemize}
\item $\Gamma_{ij}$ is the subgroup of $X^*((\CC^*)^{n-1})=\ZZ^{n-1}=\ZZ^n/(\ZZ v)$ generated by $e_{ij}/(\ZZ v)$, where $e_{ij}\in\ZZ^n$ is the vector that has $1$ in the $i$-th coordinate, $-1$ in the $j$-th coordinate, and $0$ elsewhere;
\item $\phi_{ij}$ is the trivial homomorphism that maps every element of $\Gamma_{ij}$ to $1\in\CC^*$.
\end{itemize}

An arrangement that arises from a graph as in Definition~\ref{defn_hyp_graph_arr} or Definition~\ref{defn_tor_graph_arr} is also called \emph{graphic arrangement}.

\begin{exam}
The (hyperplane or toric) arrangement of type $A_{n-1}$ is a graphic arrangement associated with $K_n$, i.e.\ the complete graph on $n$ vertices. Indeed, any graphic arrangement can be viewed as a subarrangement of it.
\end{exam}

For the remainder of the section, let $\Gamma$ be a graph on $n$ vertices. Notice that by construction the poset of intersections $L(\arh(\Gamma))$ and the poset of layers $\play(\art(\Gamma))$ are isomorphic and, in fact, we can provide a further characterization. Consider the poset of set partitions of $[n]$ ordered by refinement and denote it by $\Pi_n$.
\begin{prop}
Given a subset $I\subseteq [n]$, let $\Gamma(I)$ be the subgraph of $\Gamma$ induced by $I$. Let $\Pi_n(\Gamma)\subseteq\Pi_n$ be the subposet such that $\pi=\{\pi_1,\dotsc,\pi_k\}\in\Pi_n(\Gamma)$ iff $\Gamma(\pi_\ell)$ is connected for each $\ell=1,\dotsc,k$. The poset $L(\arh(\Gamma))$ (and consequently $\play(\art(\Gamma))$) is isomorphic to $\Pi_n(\Gamma)$.
\end{prop}
\begin{proof}
Let $\pi=\{\pi_1,\dotsc,\pi_k\}\in\Pi_n(\Gamma)$. We define $H_\pi\in L(\arh(\Gamma))$ as
\begin{equation}\label{eq:hpi}
H_\pi = \bigcap_{\ell=1}^k \left(\bigcap_{e_{ij}\in \EE(\Gamma(\pi_\ell))} H_{ij}\right)
\end{equation}
where $\EE(\Gamma(\pi_\ell))$ is the set of edges of $\Gamma(\pi_\ell)$. Since $\pi\in\Pi_n(\Gamma)$, this definition is equivalent to
\[
H_\pi=\{(x_1,\dotsc,x_n)\in \CC^n\st x_i=x_j\textrm{ if exists }\ell \textrm{ s.t.\ }\{i,j\}\subseteq \pi_\ell\}/(\CC v).
\]

On the other hand, let $H\in L(\arh(\Gamma))$. It is described by a set of equations
\[
\begin{cases}
x_{i_1}=\dotsb=x_{i_s}, \\
\dots \\
x_{j_1}=\dotsb=x_{j_t},
\end{cases}
\]
where the indices in each row belong to disjoint subsets of $[n]$. Notice that by definition two indices can belong to the same row of equations only if there is a path in $\Gamma$ connecting the corresponding vertices. Define the partition
\[
\pi_H=\{\{i_1,\dotsc,i_s\},\dotsc,\{j_1,\dotsc,j_t\}\}
\]
(eventually completed with singletons corresponding to the variables not appearing in the equations) which, by the previous remark, belongs to $\Pi_n(\Gamma)$.

Now it is an exercise to prove that the two maps
\[
\begin{array}{ccc}
L(\arh(\Gamma)) & \longrightarrow & \Pi_n(\Gamma) \\
H & \longmapsto & \pi_H
\end{array}
\qquad\textrm{and}\qquad
\begin{array}{ccc}
\Pi_n(\Gamma) & \longrightarrow & L(\arh(\Gamma))\\
\pi & \longmapsto & H_\pi
\end{array}
\]
are inverse of each other.
\end{proof}
\begin{rmk}\label{rmk:codimension}
We can compute the codimension of any $H\in L(\arh(\Gamma))$ from $\pi_H$:
\[
\codim(H)=\sum_{B\in\pi_H}\left(\card{(B)}-1\right).
\]
\end{rmk}
We now study some particular projective wonderful models associated with graphic arrangements, starting with the choice of the building set.% Notice that the well-connectedness hypothesis for the toric setting is not a concern for graphic arrangements, since the intersection of any two layers of $\play(\art(\Gamma))$ is always connected.
\begin{defn}\label{def:blockbuildingset}
Let $\GG$ be the subset of either $\play(\art(\Gamma))$ or $L(\arh(\Gamma))$, depending on the context, whose elements are the subtori/subspaces $G$ such that the corresponding partition $\pi_G$ has exactly one non-singleton block. In this case, if there is no ambiguity, we will use the symbol $\pi_G$ also to denote the only non-singleton block of the partition.
%Let $\GG$ be the set of elements $G$ such that the corresponding partition $\pi_G$ has exactly one non-singleton block. We will consider $\GG$ as a subset either of $\play(\art(\Gamma))$ or of $L(\arh(\Gamma))$ depending on the context. If there is no ambiguity, we will use the symbol $\pi_G$ also to denote the only non-singleton block of the partition.
\end{defn}
\begin{prop}\label{prop:G_building}
$\GG$ is a building set.
\end{prop}
\begin{proof}
Let $H\notin\GG$ and let $B_1,\dotsc,B_\ell$, $\ell\geq 2$, be the non-singleton blocks of the associated partition $\pi_H$. We claim that the minimal elements (w.r.t.\ the inclusion of subspaces) of $\GG_H\coloneqq\{G\in\GG\st H\subset G\}$ are $G_1,\dotsc,G_\ell$, where $G_i$ is such that $\pi_{G_i}=B_i$ for $i=1,\dotsc,\ell$. This is clear from the fact that at partition level we are looking for the coarsest partitions with only one non-singleton block that are finer than $\pi_H$.

From this it follows easily that $H=G_1\cap\dotsb\cap G_\ell$ and this intersection is transversal.
\end{proof}
%Clearly, if $G\in \GG_H$ then its associated partition $\pi_G$ must be finer than $\pi_H$. Now, asking for $G$ to be minimal in $\GG_H$ means that we want $\pi_G$ to be coarser than $\pi_{G'}$ for any other $G'\in \GG_H$. We deduce that the unique non-singleton block of $\pi_G$ must coincide with a block of $\pi_H$. 
\begin{rmk}
As shown in~\cite{gps}, in the case of $\Gamma=K_n$ this building set corresponds to the minimal one, i.e.\ the building set of irreducibles. However, this is not true in general. For example, suppose that the graph $\Gamma$ contains the edges $(1,2)$ and $(2,3)$ but not $(1,3)$: in this case the element $G\in\GG$ whose block is $\pi_G=\{1,2,3\}$ is not irreducible as a subspace.
\end{rmk}

When we consider the cone $\susp{\Gamma}$, the corresponding building set $\susp{\GG}$, defined as in Definition~\ref{def:blockbuildingset}, admits a useful description in terms of the building set $\GG$ associated with $\Gamma$. 
\begin{prop}\label{prop:building_cone_types}
Let
\[
\begin{array}{rccc}
\iota\colon & \Pi_n(\Gamma) & \longrightarrow & \Pi_{n+1}(\susp{\Gamma}) \\
& \{ \pi_1\dotsc,\pi_k\} & \longmapsto & \{ \{0\},\pi_1,\dotsc,\pi_k\}
\end{array}
\]
where we consider the elements of $\Pi_{n+1}(\susp{\Gamma})$ as partitions of $\{0,1,\dotsc,n\}$. Then each element of $\susp{\GG}\subseteq\Pi_{n+1}(\susp{\Gamma})$ is associated with a partition of one of the following forms:
\begin{enumerate}
\item $\iota(G)$ for some $G\in\GG$;
\item $\{\{0\}\cup S,\{i_1\},\dotsc,\{i_k\}\}$ for some $S\in\mathcal{P}([n])\setminus\{\emptyset\}$ where $[n]\setminus S=\{i_1,\dotsc,i_k\}$.
\end{enumerate}
\end{prop}
\begin{proof}
In fact, the partitions of type 1.\ and 2.\ above are exactly the ones with a single non-singleton block that induces a connected subgraph of $\susp{\Gamma}$.
\end{proof}
%\begin{proof}
%By definition, the partitions associated with the elements of $\susp\GG$ are the partitions of $\Pi_{n+1}$ that induce a connected subgraph of $\susp\Gamma$ and have only one non-singleton block. Let $\pi$ be such a partition and let $B$ its only non-singleton block. If $0\in B$, then it is a partition of type 2.\ above; otherwise, $\{0\}\in\pi$, so let $\pi'\in\Pi_n$ such that $\iota(\pi')=\pi$ and notice that the subgraph induced by $\pi'$ in $\susp\Gamma$ is the same as the one induced by $\pi'$ in $\Gamma$, so it is connected and $\pi'\in G$.
%\end{proof}
We now have all the ingredients to introduce projective wonderful models associated with a graphic hyperplane/toric arrangement.

%i.e. the fan whose chambers are $C_{\sigma}:=\{(x_1,\ldots,x_n)\in \CC^n| x_{\sigma(1)}<\ldots<x_{\sigma(n)}\}$ for every permutation $\sigma\in S_n$.
Let us begin with the toric case; as we have seen in Section~\ref{sec:models}, the starting point is a good toric variety. As noticed in~\cite{deconcgaiffi1}, for any graphic toric arrangement whose graph has $n$ vertices there is a canonical choice of such a variety: the one associated with the fan $\Delta$ induced by the Weyl chambers of the root system of type $A_{n-1}$. In the end, following the construction recalled in Section~\ref{sec:models}, we obtain the projective wonderful model $\Yt(\Gamma,\GG)\coloneqq\Yt(\art(\Gamma),\GG)$.

In the hyperplane setting we do not consider the arrangement $\arh(\Gamma)$ and instead we study $\arh(\susp\Gamma)$, the arrangement associated with the cone $\susp\Gamma$, and the building set $\susp\GG$ defined above, obtaining the model $\Yh(\susp\Gamma,\susp\GG)\coloneqq\Yh(\arh(\susp\Gamma),\susp\GG)$. In the next section we investigate the relationship between $\Yt(\Gamma,\GG)$ and $\Yh(\susp\Gamma,\susp\GG)$.

\section{The isomorphism between the models}\label{sec:isomorphism_models}
Let $\Gamma$ be a graph on $n$ vertices and let $\GG$ be the building set defined in Definition~\ref{def:blockbuildingset}. As spoiled by the title of this section, we want to prove the following result:
\begin{thm}\label{thm:main}
The two models $\Yt(\Gamma,\GG)$ and $\Yh(\susp\Gamma,\susp\GG)$ are isomorphic.
\end{thm}
\begin{proof}
The first step of the proof consists in noticing that the complement $\MM(\art(\Gamma))$ can be realized as the complement of a subspace arrangement in a suitable projective space.

As in the beginning of Section~\ref{sec:graphic_arr}, let $v=(1,\dotsc,1)$ and let $\EE\subseteq [n]\times[n]$ be the set of edges of $\Gamma$. We identify $(\CC^*)^{n-1}=(\CC^*)^n/(\CC^* v)$ with
\[
\PP(\CC^n)\setminus \bigcup_{i=1}^n\{ t_i=0\}
\]
where $t_1,\dotsc,t_n$ are the projective coordinates on $\PP(\CC^n)$. The map is given by
\[
\begin{array}{rcl}
(\CC^*)^n/(\CC^* v) & \longrightarrow & {\displaystyle \PP(\CC^n)\setminus \bigcup_{i=1}^n\{ t_i=0\} }\\
\lbrack(t_1,\dotsc,t_n)\rbrack & \longmapsto & [t_1,\dotsc,t_n]
\end{array}
\]
where on the left the notation with square brackets denotes the class of $(t_1,\dotsc,t_n)$ in the quotient. Under this identification, the complement $\MM(\art(\Gamma))$ is the same as the complement of an arrangement $\arr'$ in $\PP^{n-1}=\PP(\CC^n)$ whose projective hyperplanes are
\begin{enumerate}[label=(\roman*)]
\item $\{t_i=0\}$, for $i=1,\dotsc,n$;
\item $\{t_i=t_j\}$, for $(i,j)\in \EE$.
\end{enumerate}

For convenience, we recall again the construction of the model $\Yt(\Gamma,\GG)$ provided by~\cite{deconcgaiffi1}:
\begin{enumerate}
\item first, we embed the torus $(\CC^*)^n/(\CC^* v)$ in the toric variety $X_\Delta$ associated with the Coxeter fan $\Delta$ of type $A_{n-1}$;
\item then, we blow up all the closures of the layers of $\GG$ in an order that refines inclusion.
\end{enumerate}

As a consequence of the identification of $\art(\Gamma)$ with a projective arrangement in $\PP^{n-1}$, the construction just outlined is equivalent to the following:
\begin{enumerate}
\item first, we obtain from $\PP^{n-1}$ the toric variety $X_\Delta$ by blowing up the intersections of the subspaces of $\arr'$ of type (i) in any order that refines inclusion (see for example~\cites[Section~2]{batyrevblume}[Section~3]{procesi90});
\item then we blow up the (proper transforms of the) elements of $\GG$, obtained as intersections of the subspaces of type (ii), in any order that refines inclusion.
\end{enumerate}

Now we focus on the construction of the model $\Yh(\susp{\Gamma},\susp{\GG})$ for the hyperplane arrangement $\arh(\susp{\Gamma})$. At first we build its projectivization in an opportune way. Recall that $\arh(\susp{\Gamma})$ is an arrangement in an $n$-dimensional space: if we choose to view it in the space
\[
N\coloneqq \{ (0,x_1,\dotsc,x_n)\st x_i\in\CC\}\subseteq \CC^{n+1},
\]
then its projectivization in $\PP(N)$, with coordinates $[x_1,\dotsc,x_n]$ omitting the leading zero, has the following projective hyperplanes:
\begin{itemize}
\item $\{x_i=0\}$, for $i=1,\dotsc,n$;
\item $\{x_i=x_j\}$, for $(i,j)\in \EE$.
\end{itemize}
In other words, we obtain the same projective arrangement $\arr'$ as the toric case.

Let us characterise the elements of $\susp{\GG}$ in this setting. Let $\pi$ be a partition in $\Pi_{n+1}(\susp{\Gamma})$ with only one non-singleton block. If 0 does \emph{not} belong to this block, we will denote the corresponding subspace in $\susp{\GG}$ by $H_\pi$; otherwise we will denote the subspace by $M_\pi$. Notice that these two cases correspond to the subspaces of type 1. and type 2. respectively in Proposition~\ref{prop:building_cone_types}.

With this notation, we have that the elements of type $M_\pi$ correspond to the intersections of the subspaces of type (i) in the toric setting, and the elements of type $H_\pi$ correspond to the elements of $\GG$. With this in mind,  we can conclude that the toric and hyperplane models are isomorphic if we can build the model $\Yh(\susp{\Gamma},\susp{\GG})$ by blowing up in $\PP(N)\simeq \PP^{n-1}$ the subspaces of $\susp{\GG}$ in the same order as the one described in the construction of the toric model $\Yt(\Gamma,\GG)$, i.e.\ if we can reorder the elements of $\susp{\GG}=\{G_1,\dotsc,G_m\}$ in the following way:
\begin{enumerate}
\item first, we put all the subspaces of type $M_\pi$ in any order that refines inclusion (say that these are $G_1,\dotsc,G_{t_0}$);
\item then, we put all the subspaces of type $H_\pi$ in any order that refines inclusion (say that these are $G_{t_0+1},\dotsc,G_m$); notice that by construction the segment $\susp{\GG}_{t_0+1}^t\coloneqq\{G_{t_0+1},\dotsc,G_t\}$ is building for every $t=t_0+1,\dotsc,m$.
\end{enumerate}
According to~\cite[Theorem~1.3]{li}, this is an admissible ordering of the elements of $\susp{\GG}$ if the initial segment $\susp{\GG}_t=\{G_1,\dotsc,G_t\}$ is building for every $t=1,\dotsc,m$. Let us prove this, distinguishing two cases.

\noindent\framebox{Case 1.} If all the subspaces in $\susp{\GG}_t$ are of type $M_\pi$, then any intersection of elements of $\susp{\GG}_t$ is still of type $M_\pi$ and belongs to $\susp{\GG}_t$ since the ordering refines the inclusion. This implies that $\susp{\GG}_t$ is building.

\noindent\framebox{Case 2.} If at least one subspace in $\susp{\GG}_t$ is of type $H_\pi$ (notice that in this case all the subspaces of type $M_\pi$ are in $\susp{\GG}_t$), then let $L$ be the intersection of some elements of $\susp{\GG}_t$; after having removed the non-minimal subspaces, it has the form
\begin{equation}\label{eq:building_reorder_intersect}
L=L_1\cap\dotsb\cap L_k
\end{equation}
where $k\geq 1$, $L_1$ is either of type $M_{\pi_1}$ or of type $H_{\pi_1}$, all the other $L_i$'s are of type $H_{\pi_i}$ and the non-singleton blocks of $\pi_1,\dotsc,\pi_k$ are pairwise disjoint. We notice that all the $L_i$'s belong to $\susp{\GG}_t$:
\begin{itemize}
\item we already know that it is true for the subspaces of type $M_{\pi}$;
\item for the subspaces of type $H_{\pi}$, this is true because each $L_i$ of type $H_{\pi_i}$ appears as a $\susp{\GG}_{t_0+1}^t$-factor of the intersection of some elements of $\susp{\GG}_{t_0+1}^t$, and $\susp{\GG}_{t_0+1}^t$ is building.
\end{itemize}
Moreover the partition corresponding to $L$ is given by the disjoint union of the non-singleton blocks of $\pi_1,\dotsc,\pi_k$ completed with the remaining singletons, so~\eqref{eq:building_reorder_intersect} is a transversal intersection. This ends the proof that $\susp{\GG}_t$ is building.

In conclusion, we have proven that the constructions in the toric case and in the hyperplane case are the same, therefore Theorem~\ref{thm:iterblowup} can be applied and this shows that, up to isomorphism, $\Yt(\Gamma,\GG)$ and $\Yh(\susp{\Gamma},\susp{\GG})$ are the same.
\end{proof}

\section{Applications}\label{sec:examples}
In this section we present two examples of noteworthy graphic arrangements and the models associated with them. Both the examples lead to a geometric proof of an equality involving Eulerian polynomials.

\subsection{A known example: The complete graph}
Let $K_n$ be the complete graph on $n$ vertices. As already mentioned, both in the hyperplane case and in the toric one the graphic arrangement associated with it is the arrangement of Coxeter type $A_{n-1}$, which is well-known (see for example~\cite{yuzvinsky,gaiffiselecta} for the hyperplane case and~\cite{procesi90,stembridge92,stanley,dolgachevlunts} for the toric one).

In this case it is trivial to notice that $\susp{K}_n=K_{n+1}$, so a straightforward application of Theorem~\ref{thm:main} gives the following corollary, which has already been announced (without a proper proof) in~\cite[Section~5]{gps}.
\begin{crl}
The model $\Yt(K_n,\GG)$ associated with the toric arrangement of type $A_{n-1}$ is isomorphic to the model $\Yh(K_{n+1},\susp{\GG})$ associated with the hyperplane arrangement of type $A_n$.
\end{crl}

\subsection{Disjoint union of pairs of complete graphs}
Let $K_{n,m}$ be the graph obtained as the disjoint union of $K_n$ and $K_m$, namely the graph on $n+m$ vertices $\{1,\ldots,n+m\}$ such that $(i,j)$ is an edge if and only if either $i,j\leq n$ or $i,j>n$ (see Figure~\ref{fig:disjoint_graphs} for an example).
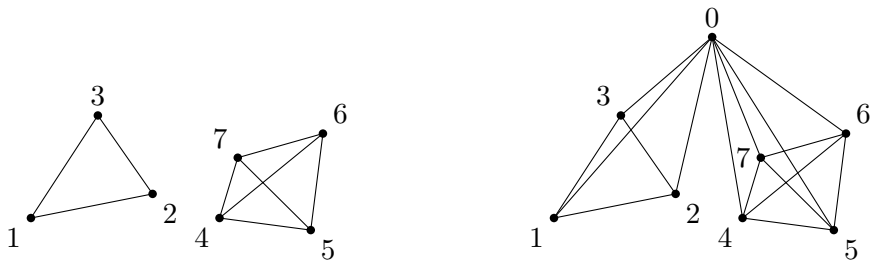
\begin{figure}[htb]
\centering
\input{images/disjoint_graphs.tikz}\hspace{2cm}\input{images/disjoint_graphs_coned.tikz}
\caption{The graph $K_{3,4}$ and its cone $\susp{K}_{3,4}$.}
\label{fig:disjoint_graphs}
\end{figure}

By Theorem~\ref{thm:main}, the toric model associated with $K_{n,m}$ is isomorphic to the projective hyperplane model associated with $\susp{K}_{n,m}$, and in particular
%By Theorem~\ref{thm:main} we know that the toric model associated with $K_{n,m}$ is isomorphic to the projective hyperplane model associated with $\susp{K}_{n,m}$, in particular%. In particular their integer cohomology rings are isomorphic
\[
H^*(\Yt(K_{n,m},\GG),\ZZ)\cong H^*(\Yh(\susp{K}_{n,m},\susp{\GG}),\ZZ).
\]
We will count the elements of two monomial bases of the two rings, obtained using the description of~\cite{gaiffiselecta,yuzvinsky} in the hyperplane case and~\cite{gps} in the toric case. To do so, we briefly recall the main tools needed for these constructions.
\begin{defn}\label{def_nested_set_simple}
Let $\Lambda$ be a (simple) arrangement of subvarieties and let $\GG$ be a building set for $\Lambda$. A subset $\NS\subseteq\GG$ is called ($\GG$-)\emph{nested} if for any set of pairwise non-comparable (w.r.t.\ the inclusion) elements $\{A_1,\dotsc,A_k\}\subseteq\NS$, with $k\geq 2$, there is an element in $\Lambda$ of which $A_1,\dotsc,A_k$ are the $\GG$-factors.
\end{defn}
In the case of graphic arrangements, considering the building set $\GG$ of Definition~\ref{def:blockbuildingset}, the nested sets are $\{H_{\pi_1},\dotsc,H_{\pi_j}\}$ such that the non-singleton blocks of $\{\pi_1,\dotsc,\pi_j\}$ are pairwise either comparable or disjoint.
\begin{defn}\label{def_admissible}
A function $f\colon\GG \to \NN$ is ($\GG$)-\emph{admissible} if it has both the following properties:
\begin{enumerate}
\item $\supp f$ is $\GG$-nested; 
\item for every $A\in \supp f$ we have $f(A)<\dim M_f(A)-\dim A$, where $M_f(A)$ is the (connected) intersection of the elements of $\supp f$ that properly contain $A$.
\end{enumerate}
Notice that the function such that $f(A)=0$ for every $A\in\GG$ is admissible.
%Notice that the zero function, i.e.\ the function such that $f(A)=0$ for every $A\in\GG$, is admissible since its support is the empty set.
\end{defn}
For each element $G\in\GG$, let $T_G$ be a polynomial variable. The \emph{admissible monomial} associated with an admissible function $f$ is
\[
m_f=\prod_{G\in \GG}T_G^{f(G)}.
\]
% Recall that the cohomology ring of a wonderful model $\Yh(\arr,\GG)$ for an arrangement of subspaces can be presented as
% \[
% H^*(\Yh(\arr,\GG),\ZZ)\cong \ZZ[T_G\st G\in\GG]/\II^H_{\GG}
% \]
% where $\II^H_{\GG}$ is a suitable ideal of relations (see~\cite[Section~5.2]{wonderful1}).
\begin{thm}[{\cite[Theorem~2.1]{gaiffiselecta}; see also~\cite{yuzvinsky}}]
The set
\[
\{m_f\st f\  \GG\textrm{-admissible}\}
\]
is a monomial basis for $H^*(\Yh(\arr,\GG),\ZZ)$.
\end{thm}

% Analogously, the cohomology ring of a wonderful model $\Yt(\arr,\GG)$ for a toric arrangement can be presented as
% \[
% H^*(\Yt(\arr,\GG),\ZZ)\cong R_{\Delta}[T_G\st G\in\GG]/\II^T_{\GG}
% \]
% where $R_{\Delta}=H^*(\Xdelta,\ZZ)$ is the integer cohomology ring of the (good) toric variety $\Xdelta$ (see Section~\ref{subsec:models_toric}), which itself can be presented as a suitable quotient of $\ZZ[C_r\st r\in\mathcal{R}]$ ($\mathcal{R}$ is the set of rays of the fan $\Delta$; see, for example,~\cite[Theorem~10.8]{danilov}), and $\II^T_{\GG}$ is a suitable ideal of relations (see~\cite[Theorem~7.1]{deconcgaiffi2}).

% In order to show a monomial basis for $H^*(\Yt(\arr,\GG),\ZZ)$, we need a couple more ingredients. Given a $\GG$-nested set $\NS$, let $\Delta(\NS)$ be a suitable subfan of $\Delta$ as defined in~\cite[Section~4]{gps} using~\cite[Theorem~5.1(2)]{deconcgaiffi2} and let $R_{\Delta(\NS)}=H^*(X_{\Delta(\NS)},\ZZ)$. Moreover, let $\pi_\NS\colon R_\Delta\to R_{\Delta(\NS)}$ be the restriction map induced by the inclusion.
% \begin{thm}[{see~\cite[Theorem~4.7]{gps}}]
% The set
% \[
% \{b\thinspace m_f\st f\  \GG\textrm{-admissible},\ b\in\Theta(\supp f)\}
% \]
% is a monomial basis for $H^*(\Yt(\arr,\GG),\ZZ)$, where $\Theta(\supp f)$ is a set of monomials in $R_\Delta$ such that $\{\pi_{\supp f}(b)\st b\in\Theta(\supp f)\}$ is a basis of $R_{\Delta(\supp f)}$.
% \end{thm}

In the toric case, let $X_\Delta$ be a good toric variety for $\arr$ and let $\mathcal{R}$ be the set of its rays; moreover, for each ray $r$, let $C_r$ be a polynomial variable.
\begin{thm}[{see~\cite[Theorem~4.7]{gps}}]
The set
\[
\{b\thinspace m_f\st f\  \GG\textrm{-admissible},\ b\in \mathcal{B}_f\}
\]
is a monomial basis for $H^*(\Yt(\arr,\GG),\ZZ)$, where $\mathcal{B}_f\subseteq\ZZ[C_r\st r\in\mathcal{R}]$ is a monomial basis of $H^*(X_{\Delta(f)},\ZZ)$ (here $\Delta(f)$ is a suitable subfan of $\Delta$ defined by $\supp f$).
\end{thm}
\begin{rmk}
Recall that for graphic toric arrangements we can choose the fan $\Delta=\Delta_{n-1}$ induced by the Weyl chambers of the root system of type $A_{n-1}$; in this case, for any $\GG$-admissible function $f$ we have that $\Delta(f)$ is isomorphic to a fan of the form $\Delta_{k-1}$ for some $k\leq n$, i.e.\ it is again a fan associated with a root system of type $A_{k-1}$. In particular we know that their Poincaré polynomials are given by Eulerian polynomials.
%Recall that for graphic toric arrangements we can always choose the fan $\Delta=\Delta_{n-1}$ induced by the Weyl chambers of the root system of type $A_{n-1}$; in this case, for any $\GG$-nested set $\NS$ we have that actually $\Delta(\NS)$ is isomorphic to a fan of the form $\Delta_{k-1}$ for some $k\leq n$, i.e.\ it is again a fan associated with a root system of type $A_{k-1}$. In particular we know that their Poincaré polynomials are given by Eulerian polynomials.
\end{rmk}

Finally, to count the elements of the monomial bases we just need to recall from~\cite{gps} the definitions of \emph{admissible trees} and \emph{admissible forests}.
%We are now almost ready to count the elements of the monomial bases for our graphic arrangements associated with $\Gamma=K_{n,m}$; we just need to recall from~\cite{gps} the definitions of \emph{admissible trees} and \emph{admissible forests}.
\begin{defn}\label{def:admissibletree}
An \emph{admissible tree} on $m$ leaves is a labeled directed rooted tree such that
\begin{itemize}
\item it has $m$ leaves, each labeled with a distinct non-zero natural number;
%\item each vertex has exactly one ingoing edge, except the root which has none;
\item each non-leaf vertex $v$ has $k_v\geq 3$ outgoing edges, and it is labeled with the symbol $q^i$ where $i\in\{1,\dotsc,k_v-2\}$.
\end{itemize}
By convention, the graph with one vertex and no edges is an admissible tree on one leaf (actually the only one). The \emph{degree} of an admissible tree is the sum of the exponents of the labels of the non-leaf vertices. Denote by $\lambda(q,t)$ the generating function of the admissible trees, i.e.\ the series whose coefficient of $q^i t^k/k!$ counts the number of admissible trees of degree $i$ on $k$ leaves (see~\cite{gaiffiselecta,yuzvinsky}).
\end{defn}
\begin{defn}
An \emph{admissible forest} on $n$ leaves is the disjoint union of admissible trees such that the sets of labels of their leaves form a partition of $\{1,\dotsc,n\}$. The \emph{degree} of an admissible forest is the sum of the degrees of its connected components.
\end{defn}
Let us consider the graphic toric arrangement associated with $K_{n,m}$ and let $\GG$ be building set of Definition~\ref{def:blockbuildingset}. Let $\art_{n-1}$ and $\art_{m-1}$ be the two toric graphic arrangements associated with $K_n$ and $K_m$ respectively; it is easy to show that every $\GG$-nested set $\NS$ can be uniquely written as the disjoint union of a $\mathcal{F}_{n-1}$-nested set $\NS_1$ and a $\mathcal{F}_{m-1}$-nested set $\NS_2$, where $\mathcal{F}_{n-1}$ and $\mathcal{F}_{m-1}$ are the ``building sets of irreducible elements'' for $\art_{n-1}$ and $\art_{m-1}$ (see~\cite[Section~5.1]{gps}), and on the other hand every such union gives rise to a $\GG$-nested set. Moreover, there is a bijection $\Psi$ between the set of $\GG$-admissible functions and the set
\[
\{(f_1,f_2)\st f_1\textrm{ is } \mathcal{F}_{n-1}\textrm{-admissible, }f_2\textrm{ is } \mathcal{F}_{m-1}\textrm{-admissible} \}
\]
given by $\Psi(f)=(f\rest{\NS_1},f\rest{\NS_2})$ if $\supp f=\NS_1\sqcup\NS_2$ as above. In other words, there is a grade-preserving bijection between the set of $\GG$-admissible functions and the set of pairs $(F_1,F_2)$ of admissible forests on $n$ and $m$ leaves respectively (where $\deg(F_1,F_2)\coloneqq \deg F_1+\deg F_2$). Finally, it is not hard to prove that, given a $\GG$-admissible function $f$ corresponding to the pair $(F_1,F_2)$, the fan $\Delta(f)$ is of the form $\Delta_{\ell-1}$, where $\ell$ is the total number of connected components of $F_1$ and $F_2$, therefore there is a bijection between $\mathcal{B}_f$ and the permutations in $S_\ell$ that is grade-preserving provided that we choose any Eulerian statistic on $S_\ell$ as the degree of a permutation. In conclusion, a basis for the $\ZZ$-cohomology of $\Yt(K_{n,m},\GG)$ is in one-to-one correspondence with the set  of triples $(F_1,F_2,\sigma)$ where $F_1$ is an admissible forest on $n$ leaves, $F_2$ is an admissible forest on $m$ leaves, and $\sigma$ is a permutation in $S_\ell$ as above. This proves that, if we define the (exponential) generating function of the family
\[
\Phi^T(q,x,y):=\sum_{n,m\geq 1}\poin(\Yt(K_{n,m},\GG),q)\frac{x^{n}}{n!}\frac{y^{m}}{m!},
\]
we have that
\begin{align}
\Phi^T(q,x,y)&{}=\sum_{\ell_1,\ell_2\geq 1}\poin(X_{\Delta_{\ell_1+\ell_2-1}},q)\frac{\lambda^{\ell_1}(q,x)}{\ell_1!}\frac{\lambda^{\ell_2}(q,y)}{\ell_2!}\nonumber\\
&{}=\sum_{\ell_1,\ell_2\geq 1}\frac{A_{\ell_1+\ell_2}(q)}{q}\frac{\lambda^{\ell_1}(q,x)}{\ell_1!}\frac{\lambda^{\ell_2}(q,y)}{\ell_2!},\label{eq:final_toric}
\end{align}
where $A_\ell(q)$ is the $\ell$-th Eulerian polynomial, defined as
\[
A_\ell(q)\coloneqq\begin{cases}
{\displaystyle \sum_{k=1}^\ell A(\ell,k)q^k}, & \ell\geq 1,\\
1, & \ell=0
\end{cases}
\]
(see~\cite{comtet}; $A(\ell,k)$ is the number of permutations in $S_\ell$ with $k-1$ descents, for $\ell\geq 1$ and $1\leq k\leq \ell$).
\begin{rmk}
In the previous formula, as well as from now on, $q$ is a polynomial variable with degree 2.
\end{rmk}

We now write a series similar to~\eqref{eq:final_toric} in the hyperplane scenario. Let $\susp{K}_{n,m}$ be the cone of $K_{n,m}$, and let $\susp{\GG}$ the building set as in Definition~\ref{def:blockbuildingset}. As usual, we identify an element $G\in\susp{\GG}$ with a subset of $\{0,1,\dotsc,n+m\}$, i.e.\ the single non-singleton block of the partition $\pi_G$. In this case the blocks relative to the elements of $\susp{\GG}$ can only be of one of the following types:
\begin{enumerate}
\item a subset of $\{1,\dotsc,n\}$;
\item a subset of $\{n+1,\dotsc,n+m\}$;
\item any subset containing $0$ of cardinality at least 2.
\end{enumerate}
In particular, in any $\susp{\GG}$-nested set $\NS$ the elements of type (3) form a linear chain with respect to the inclusion, and the set $\NS\setminus\{\textrm{elements of }\NS\textrm{ of type (3)}\}$ can be written as a disjoint union of a $\mathcal{F}_{n-1}$- and a $\mathcal{F}_{m-1}$-nested sets.

Once again, the concept of admissible forests can be used to describe the $\susp{\GG}$-admissible functions, albeit with a slight modification.
\begin{defn}
A \emph{special admissible forest} of type $(n,m)$ is an admissible forest on $n+m+1$ leaves, numbered from $0$ to $n+m$, such that
\begin{itemize}
\item its connected components not containing the leaf $0$ are admissible trees with leaf labels either contained in $\{1,\dotsc,n\}$ or contained in $\{n+1,\dotsc,n+m\}$;
\item the connected component containing the leaf $0$ has the following property: all the connected components of the graph obtained by removing the leaf $0$ and all the nodes that have 0 among their descendants are admissible trees with leaf labels either contained in $\{1,\dotsc,n\}$ or contained in $\{n+1,\dotsc,n+m\}$.
\end{itemize}
\end{defn}
In fact, a special admissible forest defines a $\susp{\GG}$-admissible function $f$ in the following way: each internal node $v$ represents an element $G_v\in\susp{\GG}$ (in particular, $\pi_{G_v}$ is the set of the labels of the leaves descending from $v$); the set $\NS=\{G_v\st v\textrm{ internal node}\}$ is a $\susp{\GG}$-nested set which is the support set of $f$; if $q^i$ is the label associated with the node $v$ then $f(G_v)=i$. % The second part of Definition~\ref{def:admissibletree} ensures that point (2) of Definition~\ref{def_admissible} is verified.
In conclusion, a basis for the $\ZZ$-cohomology of $\Yh(\susp{K}_{n,m},\susp{\GG})$ is in grade-preserving one-to-one correspondence with the set of special admissible forests of type $(n,m)$. Therefore we can study the (exponential) generating function
\[
\Phi^H(q,x,y):=\sum_{n,m\geq 1}\poin(\Yh(\susp{K}_{n,m},\susp{\GG}),q)\frac{x^{n}}{n!}\frac{y^{m}}{m!},
\]
which is equal to
\[
\sum_{n,m\geq 1}\sum_{F}q^{\deg F}\frac{x^n}{n!}\frac{y^m}{m!}
\]
where $F$ varies in the set of special admissible forests of type $(n,m)$.

Now notice that a special admissible forest of type $(n,m)$ can be obtained from two (regular) admissible forests $F_1$ and $F_2$ on $n$ and $m$ leaves respectively (where we relabel the leaves of the second one with the numbers $n+1,\dotsc,n+m$) by choosing which of their trees are ``attached'' to the leaf $0$ and how. This information is given by a permutation in $S_{\ell_1+\ell_2}$, where $\ell_1$ and $\ell_2$ are the number of trees of $F_1$ and $F_2$ (see Figure~\ref{fig:AdmissibleForests}). To show how, we recall some definitions taken mainly from~\cite{foatahan,gessel}. 

Given an ordered list of distinct numbers (not necessarily a permutation), say $\sigma=[\sigma_1,\ldots,\sigma_N]$, we denote by $\inv(\sigma)$ the set of inversions of $\sigma$:
\[
\inv(\sigma):=\{(i,j)\mid 1\leq i<j\leq N,\ \sigma_i>\sigma_j\}.
\]
\begin{defn}
A \emph{hook} is an ordered list of distinct non-zero natural numbers $\eta=[t_1,\dotsc,t_h]$, with $h\geq 2$, such that $t_1>t_2$ and $t_2<t_3<\dotsb<t_h$ (this second condition applies only for $h\geq 3$).
\end{defn}
\begin{rmk}\label{rmk:hook}
\begin{enumerate}
\item Given $s$ numbers $1\leq j_1<\dotsb<j_s\leq N$ and $i\in \{1,\ldots,s-1\}$
there is a unique way to sort $\{j_1,\ldots,j_s\}$ so that they form a hook with exactly $i$ inversions, namely $[j_{i+1},j_1,\ldots,j_i,j_{i+2},\ldots,j_s]$.
\item There exists a unique way to write a list of distinct numbers $\sigma$ as a concatenation $\sigma=p\thinspace\eta_1\dotsm\eta_k$ where each $\eta_i$ is a hook and $p$ is a list of increasing numbers. This is called the \emph{hook factorization} of $\sigma$. Notice that it is possible to have $k=0$, if $\sigma$ is an increasing sequence; also it may happen that $p=\emptyset$ ($\sigma=[3,1,2]$ is an example with $k=1$).
\end{enumerate}
\end{rmk}
\begin{defn}
Let $\sigma$ be a list of distinct numbers. The statistic \emph{lec} is defined as
\[
\lec(\sigma)=\sum_{i=1}^{k}\card{\inv(\eta_i)}
\]
where $p\thinspace\eta_1\dotsb\eta_k$ is the hook factorization of $\sigma$.
\end{defn}
\begin{lemma}
There exists a one-to-one correspondence between the set of triples $(F_1,F_2,\sigma)$ where
\begin{itemize}
\item $F_1$ is an admissible forest on $n$ leaves and with $\ell_1$ trees;
\item $F_2$ is an admissible forest on $m$ leaves and with $\ell_2$ trees;
\item $\sigma$ is a permutation in $S_{\ell_1+\ell_2}$;
\end{itemize}
and the set of special admissible forests of type $(n,m)$. This correspondence is grade-preserving provided that we define $\deg(F_1,F_2,\sigma)\coloneqq\deg(F_1)+\deg(F_2)+\lec(\sigma)$.
\end{lemma}
\begin{proof}
We only show the algorithm that associates a triple $(F_1,F_2,\sigma)$ with a special admissible forest, which is similar to the one described in~\cite{gps}.

\noindent\textbf{Preliminary step.} We define a total order on the set of trees of $F_1$ and $F_2$ in the following way: given two trees $\tau$, $\tau'$, we say that $\tau<\tau'$ if
\begin{enumerate}
\item $\tau$ belongs to $F_1$ and $\tau'$ belongs to $F_2$;
\item if $\tau$ and $\tau'$ belong to the same forest, the minimum label of the leaves in $\tau$ is smaller than the minimum label of the leaves in $\tau'$.
\end{enumerate}
Let $\{\tau_1,\dotsc,\tau_{\ell_1}\}$ and $\{\tau_{\ell_1+1},\dotsc,\tau_{\ell_1+\ell_2}\}$ be the trees of $F_1$ and $F_2$ ordered this way. Moreover, write $\sigma$ as an ordered list $[\sigma(1),\dotsc,\sigma(\ell_1+\ell_2)]$ and let $p\thinspace\eta_1\dotsb\eta_k$ be its hook factorization.

\noindent\textbf{Step 1.} Consider the last hook $\eta_k$ and let $i_1=\card{\inv(\eta_k)}$. Create a tree $\tau_0^{(1)}$ with a new internal vertex, labelled with $q^{i_1}$, to which the roots of the trees $\{\tau_j\st j\in\eta_k\}$ are attached as well as the zero-labelled leaf.

\noindent\textbf{Step 2.} Consider the second-to-last hook $\eta_{k-1}$ and let $i_2=\card{\inv(\eta_{k-1})}$. Create a tree $\tau_0^{(2)}$ with a new internal vertex, labelled with $q^{i_2}$, to which the roots of the trees $\{\tau_j\st j\in\eta_{k-1}\}$ are attached as well as the root of $\tau_0^{(1)}$.

\noindent\textbf{Next steps.} Continue with the other hooks, considering them from the last to the first. The part $p$ of the hook factorization, if present, determines the trees that will not be attached to $\tau_0$.
\end{proof}
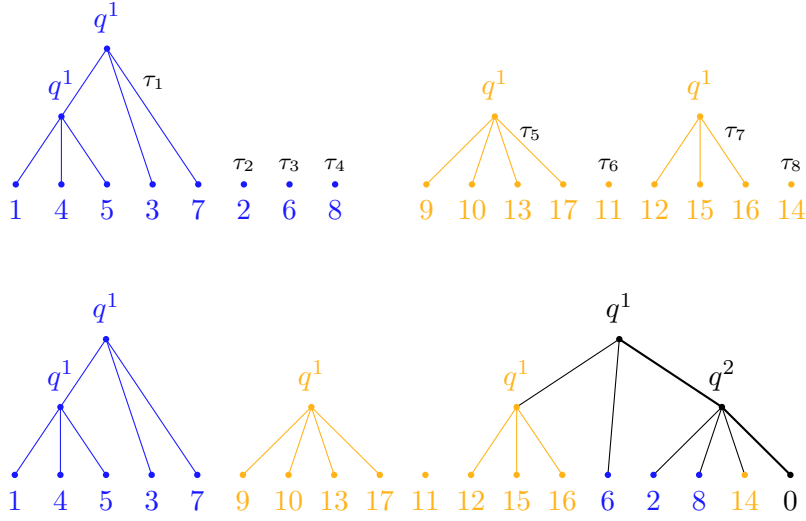
\begin{figure}
\centering
\input{images/AdmissibleForests.tikz}\\[1.5\baselineskip]
\input{images/SpecialAdmissibleForest.tikz}
\caption{Top: two admissible forests on 8 and 9 leaves respectively, each with four connected components. Bottom: the special admissible forest of type $(8,9)$ that can be obtained from the two forests above and the permutation $\sigma=[1,5,6,7,3,8,2,4]$.}
\label{fig:AdmissibleForests}
\end{figure}
The key observation is that, if we define the bivariate exponential generating function of the lec statistic
\[
\mathcal{L}(q,x,y)\coloneqq\sum_{k_1,k_2\geq 1}\left(\sum_{\sigma\in S_{k_1+k_2}}q^{\lec(\sigma)}\right)\frac{x^{k_1}}{k_1!}\frac{y^{k_2}}{k_2!},
\] 
then the coefficient of $q^dx^{\ell_1}y^{\ell_2}/(\ell_1!\ell_2!)$ in $\mathcal{L}(q,\lambda(q,x),\lambda(q,y))$ counts exactly the triples $(F_1,F_2,\sigma)$ as above and such that $\deg(F_1,F_2,\sigma)=d$. As a consequence, we have that
\begin{equation}\label{eq:final_hyper}
\Phi^H(q,x,y)=\sum_{\ell_1,\ell_2\geq 1} \sum_{\sigma\in S_{\ell_1+\ell_2}}q^{\lec(\sigma)} \frac{\lambda^{\ell_1}(q,x)}{\ell_1!}\frac{\lambda^{\ell_2}(q,y)}{\ell_2!}.
\end{equation}
By Theorem~\ref{thm:main} we know that the two series $\Phi^T(q,x,y)$ and $\Phi^H(q,x,y)$ are equal, and so are the two right-hand sides of~\eqref{eq:final_toric} and~\eqref{eq:final_hyper}. Now, $\lambda(q,t)$ is invertible with respect to the composition if viewed as a series in $\ZZ[q][[t]]$ (in fact $\lambda(q,t)=t+qt^3/3!+(q+q^2)t^4/4!+\dotsb$).
% \[
% \lambda(q,t)=t+q\frac{t^3}{3!}+(q+q^2)\frac{t^4}{4!}+\dotsb,
% \]
%its degree 0 term is zero and its degree 1 term is invertible in $\ZZ[q]$).
By composing by $\lambda^{-1}$ twice both~\eqref{eq:final_toric} and~\eqref{eq:final_hyper} we finally obtain
\[
\frac{A_{\ell}(q)}{q}=\sum_{\sigma\in S_\ell}q^{\lec(\sigma)}.
\]
%In other words, the isomorphism between the two models $\Yt(K_{n,m},\GG)$ and $\Yh(\susp{K}_{n,m},\susp{\GG})$ provides yet another proof of the fact that the lec statistic is Eulerian.
%
%In other words, through the isomorphism between the two models $\Yt(K_{n,m},\GG)$ and $\Yh(\susp{K}_{n,m},\susp{\GG})$, one can prove that the lec statistic is Eulerian.
%
In other words, the isomorphism between $\Yt(K_{n,m},\GG)$ and $\Yh(\susp{K}_{n,m},\susp{\GG})$ provides yet another proof of the fact that the lec statistic is Eulerian.

\section*{Acknowledgements}
All of the authors thank the Italian National Group for Algebraic and Geometric Structures and their Applications (GNSAGA--INdAM) for its support to their research. The work of V.~Siconolfi has also been partially funded by the European Union under the Italian National Recovery and Resilience Plan (NRRP) of NextGenerationEU, partnership on ``Telecommunications of the Future'' (PE00000001 -- program ``RESTART'', CUP: D93C22000910001) and by the Italian Ministry of University and Research under the Programme ``Department of Excellence'' Legge 232/2016 (CUP: D93C23000100001).

\printbibliography

\end{document}

%% file: images/disjoint_graphs.tikz
\begin{tikzpicture}[scale=0.8]
\coordinate (1) at (0,0);
\coordinate (2) at (2,0.4);
\coordinate (3) at (1.1,1.7);
\coordinate (4) at (3.1,0);
\coordinate (5) at (4.6,-0.2);
\coordinate (6) at (4.8,1.4);
\coordinate (7) at (3.4,1);

%\coordinate (0) at (2.6,3);

\draw (1) -- (2) -- (3) -- cycle;
\fill (1) circle [radius=0.07] node[below left] {$1$};
\fill (2) circle [radius=0.07] node[below right] {$2$};
\fill (3) circle [radius=0.07] node[above] {$3$};

\draw (4) -- (5) -- (7) -- (6) -- (4) -- (7); \draw (5) -- (6);
\fill (4) circle [radius=0.07] node[below left] {$4$};
\fill (5) circle [radius=0.07] node[below right] {$5$};
\fill (6) circle [radius=0.07] node[above right] {$6$};
\fill (7) circle [radius=0.07] node[above left] {$7$};

%\foreach \x in {1,...,7}{\draw (\x) -- (0);}
%\fill (0) circle [radius=0.07] node[above] {$0$};
\end{tikzpicture}

%% file: images/disjoint_graphs_coned.tikz
\begin{tikzpicture}[scale=0.8]
\coordinate (1) at (0,0);
\coordinate (2) at (2,0.4);
\coordinate (3) at (1.1,1.7);
\coordinate (4) at (3.1,0);
\coordinate (5) at (4.6,-0.2);
\coordinate (6) at (4.8,1.4);
\coordinate (7) at (3.4,1);

\coordinate (0) at (2.6,3);

\draw (1) -- (2) -- (3) -- cycle;
\fill (1) circle [radius=0.07] node[below left] {$1$};
\fill (2) circle [radius=0.07] node[below right] {$2$};
\fill (3) circle [radius=0.07] node[above left] {$3$};

\draw (4) -- (5) -- (7) -- (6) -- (4) -- (7); \draw (5) -- (6);
\fill (4) circle [radius=0.07] node[below left] {$4$};
\fill (5) circle [radius=0.07] node[below right] {$5$};
\fill (6) circle [radius=0.07] node[above right] {$6$};
\fill (7) circle [radius=0.07] node[left] {$7$};

\foreach \x in {1,...,7}{\draw (\x) -- (0);}
\fill (0) circle [radius=0.07] node[above] {$0$};
\end{tikzpicture}

%% file: images/AdmissibleForests.tikz
\begin{tikzpicture}[scale=0.6]
\def\quota{1.5}
\definecolor{graph1}{RGB}{25,25,255}
\definecolor{graph2}{RGB}{255,179,25}

\foreach \x/\lab in {1/1,2/4,3/5,4/3,5/7,6/2,7/6,8/8} {
\fill[graph1] (\x,0) circle [radius=0.07] node[below=0.5ex] {$\lab$};
}
\draw[graph1] (2,\quota) -- (1,0);
\draw[graph1] (2,\quota) -- (2,0);
\draw[graph1] (2,\quota) -- (3,0);
\draw[graph1] (3,2*\quota) -- (2,\quota);
\draw[graph1] (3,2*\quota) -- (4,0);
\draw[graph1] (3,2*\quota) -- (5,0);
\fill[graph1] (2,\quota) circle [radius=0.07] node[above=0.5ex] {$q^1$};
\fill[graph1] (3,2*\quota) circle [radius=0.07] node[above=0.5ex] {$q^1$};

\foreach \x/\lab in {1/9,2/10,3/13,4/17,5/11,6/12,7/15,8/16,9/14} {
\fill[graph2] (9+\x,0) circle [radius=0.07] node[below=0.5ex] {$\lab$};
}
\draw[graph2] (11.5,\quota) -- (10,0);
\draw[graph2] (11.5,\quota) -- (11,0);
\draw[graph2] (11.5,\quota) -- (12,0);
\draw[graph2] (11.5,\quota) -- (13,0);
\draw[graph2] (16,\quota) -- (15,0);
\draw[graph2] (16,\quota) -- (16,0);
\draw[graph2] (16,\quota) -- (17,0);
\fill[graph2] (11.5,\quota) circle [radius=0.07] node[above=0.5ex] {$q^1$};
\fill[graph2] (16,\quota) circle [radius=0.07] node[above=0.5ex] {$q^1$};

\node[right=1ex,font=\footnotesize,inner sep=0pt] at (3.5,1.5*\quota) {$\tau_1$};
\node[above=1ex,font=\footnotesize,inner sep=0pt] at (6,0) {$\tau_2$};
\node[above=1ex,font=\footnotesize,inner sep=0pt] at (7,0) {$\tau_3$};
\node[above=1ex,font=\footnotesize,inner sep=0pt] at (8,0) {$\tau_4$};
\node[right=1ex,font=\footnotesize,inner sep=0pt] at (11.75,0.75*\quota) {$\tau_5$};
\node[above=1ex,font=\footnotesize,inner sep=0pt] at (14,0) {$\tau_6$};
\node[right=1ex,font=\footnotesize,inner sep=0pt] at (16.25,0.75*\quota) {$\tau_7$};
\node[above=1ex,font=\footnotesize,inner sep=0pt] at (18,0) {$\tau_8$};
\end{tikzpicture}

%% file: images/SpecialAdmissibleForest.tikz
\begin{tikzpicture}[scale=0.6]
\def\quota{1.5}
\definecolor{graph1}{RGB}{25,25,255}
\definecolor{graph2}{RGB}{255,179,25}

\draw[graph1] (2,\quota) -- (1,0);
\draw[graph1] (2,\quota) -- (2,0);
\draw[graph1] (2,\quota) -- (3,0);
\draw[graph1] (3,2*\quota) -- (2,\quota);
\draw[graph1] (3,2*\quota) -- (4,0);
\draw[graph1] (3,2*\quota) -- (5,0);
\fill[graph1] (2,\quota) circle [radius=0.07] node[above=0.5ex] {$q^1$};
\fill[graph1] (3,2*\quota) circle [radius=0.07] node[above=0.5ex] {$q^1$};

\draw[graph2] (7.5,\quota) -- (6,0);
\draw[graph2] (7.5,\quota) -- (7,0);
\draw[graph2] (7.5,\quota) -- (8,0);
\draw[graph2] (7.5,\quota) -- (9,0);
\fill[graph2] (7.5,\quota) circle [radius=0.07] node[above=0.5ex] {$q^1$};

\draw[graph2] (12,\quota) -- (11,0);
\draw[graph2] (12,\quota) -- (12,0);
\draw[graph2] (12,\quota) -- (13,0);
\draw (16.5,\quota) -- (15,0);
\draw (16.5,\quota) -- (16,0);
\draw (16.5,\quota) -- (17,0);
\draw[thick] (16.5,\quota) -- (18,0);
\draw (14.25,2*\quota) -- (12,\quota);
\draw (14.25,2*\quota) -- (14,0);
\draw[thick] (14.25,2*\quota) -- (16.5,\quota);
\fill[graph2] (12,\quota) circle [radius=0.07] node[above=0.5ex] {$q^1$};
\fill (16.5,\quota) circle [radius=0.07] node[above=0.5ex] {$q^2$};
\fill (14.25,2*\quota) circle [radius=0.07] node[above=0.5ex] {$q^1$};

\fill[graph1] (1,0) circle [radius=0.07] node[below=0.5ex] {$1$};
\fill[graph1] (2,0) circle [radius=0.07] node[below=0.5ex] {$4$};
\fill[graph1] (3,0) circle [radius=0.07] node[below=0.5ex] {$5$};
\fill[graph1] (4,0) circle [radius=0.07] node[below=0.5ex] {$3$};
\fill[graph1] (5,0) circle [radius=0.07] node[below=0.5ex] {$7$};
\fill[graph2] (6,0) circle [radius=0.07] node[below=0.5ex] {$9$};
\fill[graph2] (7,0) circle [radius=0.07] node[below=0.5ex] {$10$};
\fill[graph2] (8,0) circle [radius=0.07] node[below=0.5ex] {$13$};
\fill[graph2] (9,0) circle [radius=0.07] node[below=0.5ex] {$17$};
\fill[graph2] (10,0) circle [radius=0.07] node[below=0.5ex] {$11$};
\fill[graph2] (11,0) circle [radius=0.07] node[below=0.5ex] {$12$};
\fill[graph2] (12,0) circle [radius=0.07] node[below=0.5ex] {$15$};
\fill[graph2] (13,0) circle [radius=0.07] node[below=0.5ex] {$16$};
\fill[graph1] (14,0) circle [radius=0.07] node[below=0.5ex] {$6$};
\fill[graph1] (15,0) circle [radius=0.07] node[below=0.5ex] {$2$};
\fill[graph1] (16,0) circle [radius=0.07] node[below=0.5ex] {$8$};
\fill[graph2] (17,0) circle [radius=0.07] node[below=0.5ex] {$14$};
\fill (18,0) circle [radius=0.07] node[below=0.5ex] {$0$};
\end{tikzpicture}